\documentclass[12pt]{amsart}

\usepackage{fullpage}
\usepackage{amsmath}
\usepackage{amsthm}
\usepackage{amssymb}
\usepackage{amsfonts,mathrsfs}
\usepackage{amsxtra}
\usepackage{epsfig}
\usepackage{verbatim}
\usepackage{color}
\usepackage{hyperref}
\usepackage{longtable}

\theoremstyle{plain}
\newtheorem{theorem}[equation]{Theorem}
\newtheorem{proposition}[equation]{Proposition}
\newtheorem{lemma}[equation]{Lemma}

\theoremstyle{remark}

\newcommand{\BPol}{\mathbf{B}_{\Omega}^{\lambda}}

\newcommand{\BKol}{B_{\Omega}^{\lambda}}

\newcommand{\Ltol}{L^2(\Omega,\lambda)}
\newcommand{\Ltola}{L^2_a(\Omega,\lambda)}
\newcommand{\Lpol}{L^p(\Omega,\lambda)}

\newcommand{\exprho}{\exp\left(\frac{1}{\rho}\right)}

\newcommand{\BPdm}{\mathbf{B}_{\mathbb{D}}^{\mu}}
\newcommand{\BKdm}{B_{\mathbb{D}}^{\mu}}
\allowdisplaybreaks

\begin{document}

\title[The Bergman Projection Operator On Reinhardt Domains]{Mapping Properties of Weighted Bergman Projection Operators on Reinhardt Domains}
\author{\v Zeljko \v Cu\v ckovi\'c}
\address[\v Zeljko \v Cu\v ckovi\'c]{University of Toledo, Department of Mathematics and Statistics, Toledo, OH 43606}
\email{zeljko.cuckovic@utoledo.edu}

\author{Yunus E. Zeytuncu}
\address[Yunus E. Zeytuncu]{University of Michigan - Dearborn, Department of Mathematics and Statistics, Dearborn, MI 48128}
\email{zeytuncu@umich.edu}


\date{}
\begin{abstract}
We show that on smooth complete Reinhardt domains, weighted Bergman projection operators corresponding to exponentially decaying weights are unbounded on $L^p$ spaces for all $p\not=2$. On the other hand, we also show that  the exponentially weighted projection operators are bounded on Sobolev spaces on the unit ball.
\end{abstract}
\maketitle
\section{Introduction}

Let $\Omega$ be a bounded domain in $\mathbb{C}^n$ and let $\lambda$ be a positive continuous function on $\Omega$. We denote the standard Lebesgue measure by $dV(z)$ and we consider $\lambda$ as a weight function on $\Omega$. The space of square integrable functions on $\Omega$ with respect to the measure $\lambda(z)dV(z)$ is denoted by $\Ltol$. The weighted inner product and the corresponding norm are defined in the usual way. The space of square integrable holomorphic functions on $\Omega$ is denoted by $\Ltola$. Since the weight is continuous and positive, $\Ltola$ is a closed subspace of $\Ltol$ (see \cite{Pasternak90}). The weighted Bergman projection operator $\BPol$ is the orthogonal projection operator from $\Ltol$ onto $L^2_a(\Omega,\lambda)$. By the Riesz representation theorem, $\BPol$ is an integral operator of the form 

$$\BPol f(z)=\int_{\Omega}\BKol (z,w)f(w)\lambda(w)dV(w)$$
where the kernel $\BKol (z,w)$ is called the weighted Bergman kernel. As an integral operator, mapping properties of $\BPol$ on $\Lpol$ for $p\not=2$ are also of interest. In particular, for a given pair  $\Omega$ and $\lambda$, one can ask  the question for which $p\in(1,\infty)$ is the operator $\BPol$ bounded on $L^p(\Omega,\lambda)$.

This question has been addressed in many settings and we refer the reader to \cite{LanzaniStein04, Ligocka89, Charp06, KrantzPeloso07Statements, ChangLi, SonmezBarrett} and the references therein. One of the well studied settings is the case when $\Omega$ is the unit disk $\mathbb{D}$ in $\mathbb{C}^1$. Two extreme cases stand out on $\mathbb{D}$. Let $\rho(z)=|z|^2-1$ be the standard defining function\footnote{Both of the results hold for more general defining functions.} for $\mathbb{D}$.

\begin{itemize}
\item (Polynomial decay \cite{ForelliRudin}) If $\lambda=(-\rho)^k$ for some $k>0$ then the weighted Bergman projection operator $\mathbf{B}_{\mathbb{D}}^{\lambda}$ is bounded on $L^p(\mathbb{D},\lambda)$ for all $p\in (1,\infty)$.

\item (Exponential decay  \cite{Dostanic04, ZeytuncuTran}) On the other hand, if $\lambda=\exp\left(\frac{1}{\rho}\right)$  then the weighted Bergman projection operator $\mathbf{B}_{\mathbb{D}}^{\lambda}$ is bounded on $L^p(\mathbb{D},\lambda)$ if and only if $p=2$. 
\end{itemize}
The change in $\lambda$, from polynomial decay to exponential decay, causes a drastic change in $L^p$ boundedness of weighted Bergman projections.

We are interested in whether the same phenomenon happens on domains in higher dimensions when we change the weight from polynomial to exponential decay.

\subsection{$L^p$  irregularity} We start our investigation on smooth complete Reinhardt domains. In this setting, the polynomial decay case (under some additional conditions on $\Omega$) follows from a more general theorem in  \cite{Charpentier13, Charpentier14}. Namely, if  $\Omega$ is a smooth bounded Reinhardt pseudoconvex domain of finite type in $\mathbb{C}^2$, $\rho$ is a defining function for $\Omega$ and $\lambda=(-\rho)^q$ where $q$ is a non-negative rational number then the  weighted Bergman projection operator $\mathbf{B}_{\Omega}^{\lambda}$ is bounded on $L^p(\Omega,\lambda)$ for all $p\in(1,\infty)$. See \cite{Charpentier13, Charpentier14}  and also \cite{BonamiGrellier} for the case $q$ is a natural number and \cite{McNSte94} for the case $q=0$. 

This result is analogous to the first result on the unit disk; that is, the weight decays at a polynomial rate on the boundary. Hence, it is natural to ask what happens if we take an exponentially decaying weight $\lambda=\exprho$ in this setting. We present an answer in the following theorem.

\begin{theorem}\label{main}
Let $\Omega$ be a smooth bounded complete Reinhardt domain in $\mathbb{C}^2$ and let $\rho$ be a smooth multi-radial defining function for $\Omega$. If $\lambda=\exprho$ then the weighted Bergman projection $\BPol$ is bounded on $\Lpol$ if and only if $p=2$.
\end{theorem}

Here, by a multi-radial function $\rho(z_1, z_2)$ we mean a function that depends on $|z_1|$ and $|z_2|$, i.e. $\rho(z_1,z_2)=\rho(|z_1|,|z_2|)$. 
The same result holds on smooth bounded complete Reinhardt domains in $\mathbb{C}^n$ for $n\geq 3$. However, we will present the proof below only in two dimensions to avoid clumsy notation.

\subsection{Sobolev regularity on the unit ball} Let $\mathbb{B}^n$ denote the unit ball in $\mathbb{C}^n$ and $r(z_1,\cdots,z_n)=|z_1|^2+\cdots+|z_n|^2-1$ be the standard defining function for $\mathbb{B}^n$. If we let $\mu(z)=\exp\left(\frac{1}{r}\right)$ then by Theorem \ref{main} the weighted Bergman projection $\mathbf{B}^{\mu}_{\mathbb{B}^n}$ is not bounded on $L^p(\mathbb{B}^n,\mu)$ for $p\not=2$.
Even though $\mathbf{B}^{\mu}_{\mathbb{B}^n}$ exhibits this behavior on $L^p$ spaces, its behavior is more predictable on $L^2$-Sobolev spaces. In particular, Sobolev regularity of $\mathbf{B}^{\mu}_{\mathbb{B}^n}$ is not sensitive to decay rates. It is also noted in \cite{Charpentier13, Charpentier14} that the weighted Bergman projection corresponding to the polynomially decaying weight $(-r)^k$ is bounded on Sobolev spaces. 

For $k\in\mathbb{N}$, we denote by $W^k(\mathbb{B}^n,\mu)$ the weighted $L^2$-Sobolev space with the norm 
$$||f||_{k,\mu}^2=\sum_{|\beta+\gamma|\leq k}\int_{\mathbb{B}^n}\left|\frac{\partial^{\beta+\gamma}}{\partial \overline{z}^{\beta}\partial z^{\gamma}}f(z)\right|^2\mu(z)dV(z).$$
For the exponentially decaying weight $\mu$ on the unit ball, we prove the following boundedness result on weighted $L^2$-Sobolev spaces.

\begin{theorem}\label{Sobolev}
Let $\mathbb{B}^n, r(z)$ and $\mu$ be as above. For any $k\in \mathbb{N}$, the weighted Bergman projection $\mathbf{B}^{\mu}_{\mathbb{B}^n}$ is bounded on $W^k(\mathbb{B}^n,\mu)$.
\end{theorem}

It is a natural question to investigate the similar result for non-integer values of $k$. This boils down studying the interpolation properties of Sobolev spaces with respect to exponential weights. We postpone this investigation to a future project. 

The unweighted version of this theorem on general smooth complete Reinhardt domains is in \cite{Boas84} and \cite{Straube86}. We imitate the second proof in \cite{Boas84}; however, instead of a Brunn-Minkowski type inequality we use an asymptotic estimate. It will be clear in the proof that generalizing Theorem \ref{Sobolev} to general Reinhardt domains with general exponentially decaying weights is not immediate.

We write $A\lesssim B$ to mean that there exists a uniform constant $c>0$ such that $A\leq cB$. We also use $A\approx B$ to mean that there exists a uniform constant $k>0$ such that $\frac{1}{k}A\leq B\leq kA$.



\section{Proof of Theorem \ref{main}}

\subsection{Preliminaries} Since $\Omega$ is a bounded Reinhardt domain, its projection onto $z_1$ axis is a disk and without loss of generality we can assume that this projection is the unit disk. We denote the radial image of $\Omega$ (that is the image in the $|z_1|$ and $|z_2|$ plane) by $R \subset \mathbb{R}^2$.

For a point $z_1\in\mathbb{D}$, let $S_{z_1}$ denote the slice of $\Omega$ over $z_1$ in the $z_2$ direction. Each $S_{z_1}$ is a disk. We define the following auxiliary weight on $\mathbb{D}$,

\begin{equation}\label{mu}
\mu(z_1)=\text{ area of the disk }S_{z_1}=\int_{S_{z_1}}\exprho dA(z_2)
\end{equation}
where $\rho$ is also restricted to $S_{z_1}$ in the first coordinate. It is easy to notice that $\mu$ is a multi-radial function of $z_1$.

The set of monomials $\{z^{\alpha}\}$ (for all $\alpha\in\mathbb{N}^2$) forms an orthogonal basis for the Bergman space $L^2_a(\Omega,\lambda)$. When we set coefficients
\begin{align*}
c_{\alpha}^2=\frac{1}{\int_{\Omega}|z^{\alpha}|^2\lambda(z)dV(z)},
\end{align*}
we get the following representation of the weighted Bergman kernel (see \cite{KrantzSCVbook})
\begin{align*}
\BKol(z,w)=\sum_{\alpha}c^2_{\alpha}z^{\alpha}\overline{w}^{\alpha}.
\end{align*}

\subsection{Reduction to $\BPdm$}

We start with a relation between the weighted Bergman kernel $\BKol$ (defined on $\Omega\times\Omega$) and the weighted Bergman kernel $\BKdm$ (defined on $\mathbb{D}\times\mathbb{D}$).

\begin{lemma}\label{four}
For a fixed $z\in\Omega$,
\begin{align}\label{kernels}
\int_{S_{w_1}}\BKol(z,w)\lambda(w)dA(w_2)
=\mu(w_1)\BKdm (z_1,w_1).
\end{align}
\end{lemma}

\begin{proof}
Since $\Omega$ is Reinhardt, the Bergman kernel has the symmetry property that $$\BKol(z,w)=\BKol(tz,t^{-1}w)$$ for $t>0$ such that $tz$ and $t^{-1}w$ are in $\Omega$. Hence when $z$ is fixed, we can bump $z$ out little bit and bump $w$ in so that $w$ is confined to a compact set. Now for $w_1\in\mathbb{D}$ we have 
\begin{align*}
\int_{S_{w_1}}\BKol(z,w)\lambda(w)dA(w_2)&=\sum_{\alpha}c^2_{\alpha}z^{\alpha}\overline{w_1}^{\alpha_1}\int_{S_{w_1}}\overline{w_2}^{\alpha_2}\exprho dA(w_2).
\end{align*}
The change in order of summation and integration holds since for a fixed $z$, we can keep $w$ in a compact set by the symmetry property above and the sum on the right hand side converges uniformly.
The function $\exprho\vert_{w_1}$ is radially symmetric in $w_2$. Therefore, only $\alpha_2=0$ contributes a non-zero integral above and we get
\begin{align*}
\int_{S_{w_1}}\BKol(z,w)\lambda(w)dA(w_2)&=\sum_{\alpha_1=0}^{\infty}c^2_{(\alpha_1,0)}z_1^{\alpha_1}\overline{w_1}^{\alpha_1}\int_{S_{w_1}}\exprho dA(w_2)\\
&=\mu(w_1)\sum_{\alpha_1=0}^{\infty}c^2_{(\alpha_1,0)}z_1^{\alpha_1}\overline{w_1}^{\alpha_1}.
\end{align*}
We notice that 
\begin{align*}
c^2_{(\alpha_1,0)}&=\frac{1}{\int_{\Omega}|z_1^{\alpha_1}|^2\lambda(z)dV(z)}=\frac{1}{\int_{\mathbb{D}}|z_1^{\alpha_1}|^2\mu(z_1)dA(z_1)},
\end{align*}
and conclude that 
\begin{align*}
\int_{S_{w_1}}\BKol(z,w)\lambda(w)dA(w_2)
=\mu(w_1)\BKdm (z_1,w_1).
\end{align*}

\end{proof}

This observation enables us to prove the following relation between $L^p$ mapping properties of weighted Bergman projections.

\begin{proposition}\label{one}
If the weighted Bergman projection $\mathbf{B}_{\mathbb{D}}^{\mu}$ is unbounded on $L^p(\mathbb{D},\mu)$ for some $p\geq 1$, then so is $\BPol$ on $\Lpol$.
\end{proposition}

\begin{proof}
Since the operator is unbounded there exists a sequence of functions $\left\{f_n\right\}\subset L^p(\mathbb{D},\mu)\cap L^2(\mathbb{D},\mu)$ such that 
\begin{align}\label{ratio1}
\lim_{n\to\infty}\frac{||\BPdm f_n||_{\mathbb{D},\mu}}{||f_n||_{\mathbb{D},\mu}}=\infty.
\end{align}
The norms are $L^p$ norms but we drop $p$ to simplify the notation. We define a new sequence of functions $\left\{F_n\right\}\subset \Lpol$ by $F_n(z_1,z_2)=f_n(z_1)$. Then we get \begin{align}\label{norms1}
||F_n||_{\Omega,\lambda}=||f_n||_{\mathbb{D},\mu}.
\end{align}

On the other hand, by using \eqref{kernels} we notice that for any $n\in\mathbb{N}$

\begin{align}\label{norms2}
\BPol(F_n)(z_1,z_2)=\BPdm(f_n)(z_1),
\end{align}
and by using \eqref{norms1}
\begin{align}\label{norms3}
||\BPol(F_n)||_{\Omega,\lambda}=||\BPdm(f_n)||_{\mathbb{D},\mu}.
\end{align}
We finish the proof of the proposition by combining \eqref{ratio1}, \eqref{norms1} and \eqref{norms3}.

\end{proof}

\subsection{Analysis of $\BPdm$}
Next we study $L^p$ mapping properties of the operator $\BPdm$ and prove the following.

\begin{proposition}\label{two}
For $1<p<\infty$, $\mathbf{B}_{\mathbb{D}}^{\mu}$ is bounded on $L^p(\mathbb{D},\mu)$  if and only if $p=2$.
\end{proposition}

The arguments are similar to the ones in \cite{ZeytuncuTran}; however, the weights in \cite{ZeytuncuTran} were explicit whereas the weight here is not. The weight on the unit disc here is the restriction of $\lambda$, which involves an unknown defining function in it.

\begin{proof} It is enough to prove the statement for $1<p<2$. We get the remaining part by using duality of $L^p$ spaces and self adjointness of $\BPdm$.

\subsubsection*{Moment Function} We define the following moment function (and its logarithm) for $x>0$,
\begin{align*}
\Phi(x)&=\int_{\mathbb{D}}|z|^x\mu(z)dA(z)\\
&=e^{\phi(x)}.
\end{align*}
After successive integration by parts $n$-times (due to exponential decay at the boundary we don't get any boundary terms), we rewrite it as
\begin{align*}
\Phi(x)&=\int_{\mathbb{D}}|z|^x\mu(z)dA(z)\\
&=\int_{\Omega}|z_1|^x\exprho dV(z_1,z_2)\\
& =4\pi^2 \int_{R} r_1^{x+1}r_2\exprho dr_1dr_2\\
&=\frac{4\pi^2}{(x+2)\cdots(x+n+1)}\int_{R}r_1^{x+n+1}(-1)^n\frac{\partial^n}{\partial r_1^n}\left(r_2\exprho\right)dr_1dr_2.\\
\end{align*}

We label the last integral and the part of the integrand as
\begin{align*}
\delta_n(r_1,r_2)&=(-1)^n\frac{\partial^n}{\partial r_1^n}\left(r_2\exprho\right)\\
\Phi_n(x)&=\int_{R}r_1^{x+n+1}(-1)^n\frac{\partial^n}{\partial r_1^n}\left(r_2\exprho\right)dr_1dr_2\\
&=\int_{R}r_1^{x+n+1}\delta_n(r_1,r_2)dr_1dr_2.
\end{align*}

A straightforward computation expresses $\delta_n$'s as follows,
\begin{align*}
\delta_1(r_1,r_2)&=\left[\rho_{r_1}\right]\frac{r_2}{\rho^2}\exprho
\\
\delta_2(r_1,r_2)&=\left[(\rho_{r_1})^2+\rho(2(\rho_{r_1})^2-\rho_{r_1r_1}\rho)\right]\frac{r_2}{\rho^4}\exprho\\
&\vdots\\
\delta_n(r_1,r_2)&=\left[(\rho_{r_1})^n+\rho(\cdots)\right]\frac{r_2}{\rho^{2n}}\exprho.
\end{align*}
Without loss of generality, $\rho_{r_1}$ can be assumed to be bounded away from zero near $r_1=1$ since at $r_1=1$ the tangent line is vertical. Moreover, $\rho$ is negative inside, zero on the boundary, and positive outside; therefore the nonzero value of $\rho_{r_1}$ has to be positive. In a small neighborhood of the boundary of $R$, where $\rho$ is sufficiently small, only the first term $(\rho_{r_1})^n$ in the square bracket matters since all the other terms contain a $\rho$ term in front.  That is, we obtain the following.

\begin{lemma}\label{twist}
For any integer $n$ there exists $a_n>0$ such that $\delta_n(r_1,r_2)>c_n>0$ for some $c_n$ when $a_n< r_1<1$ and for all $r_2>0$.
\end{lemma}

Next we define 
\begin{align*}
\widetilde{\Phi}_n(x)
&=\int_{R\cap \{a_n<r_1<1\}}r_1^{x+n+1}\delta_n(r_1,r_2)dr_1dr_2.
\end{align*}

Then we compare $\widetilde{\Phi}_n(x)$ and $\Phi_n(x)$.

\begin{align*}
\left|\frac{\Phi_n(x)}{\widetilde{\Phi}_n(x)}-1\right|&\leq\frac{\int_{R\cap \{r_1<a_n\}}r_1^{x+n+1}|\delta_n(r_1,r_2)|dr_1dr_2}{\int_{R\cap \{a_n<r_1<1\}}r_1^{x+n+1}\delta_n(r_1,r_2)dr_1dr_2}\\
&\lesssim\frac{\max_{R}|\delta_n|}{\int_{R\cap \{a_n<r_1<1\}}\delta_n(r_1,r_2)dr_1dr_2}~\frac{a_n^{x+n+2}}{a_n^{x+n+1}(x+n+2)}\\
&\leq C(n)\frac{1}{x+n+2}
\end{align*}

Hence, we conclude the following lemma. 
\begin{lemma}\label{tilde}
For any $n\geq1$, $$\lim_{x\to\infty}\frac{\Phi_n(x)}{\widetilde{\Phi}_n(x)}=1.$$
\end{lemma}

We define two more functions $$\Theta_n(x)=\frac{1}{(x+2)\cdots(x+n+1)}\widetilde{\Phi}_n(x)$$
and $$\theta_n(x)=\log\Theta_n(x).$$

By Lemma \ref{tilde} for any $n\geq1$, 
\begin{equation}\label{phi-theta}
\lim_{x\to\infty}\frac{\Theta_n(x)}{\Phi(x)}=1
\end{equation}

\subsubsection*{Test Functions} 
We fix once for all a positive integer $k$ so large that $p(k+1)<2(k-1)$. This is the step where we use the assumption  $p<2$.
Then we consider the action of $\BPdm$ on a sequence of monomials $\left\{z^{km}\overline{z}^m\right\}_{m=0}^{\infty}$. By using the radial symmetry, we obtain 
\begin{align*}
\BPdm\left(z^{km}\overline{z}^m\right)(z)=\frac{\Phi(2km)}{\Phi(2(k-1)m)}z^{(k-1)m}.
\end{align*}
In order to prove unboundedness, we consider the ratio
\begin{align*}
\frac{||\BPdm\left(z^{km}\overline{z}^m\right)||_{\mathbb{D},\mu}^p}{||z^{km}\overline{z}^m||_{\mathbb{D},\mu}^p}&=\left(\frac{\Phi(2km)}{\Phi(2(k-1)m)}\right)^p\frac{\Phi(p(k-1)m)}{\Phi(p(k+1)m)}\\
&=\exp\left(p\phi(2km)-p\phi(2(k-1)m)-\left(\phi(p(k+1)m)-\phi(p(k-1)m)\right)\right).
\end{align*}
Our goal is to show that this expression gets as large as we want by choosing appropriate $m$. \\

In order to achieve this, we first look at another expression. We set
$$R_n(m)=p\theta_n(2km)-p\theta_n(2(k-1)m)-\left(\theta_n(p(k+1)m)-\theta_n(p(k-1)m)\right).$$


We note that the functions $\Phi(x),\phi(x),\Theta_n(x)$ and $\theta_n(x)$ are smooth functions of $x$. We use the mean value theorem twice to rewrite $R_n(m)$ as follows
\begin{align*}
R_n(m)&=2pm\theta_n'(m_1)-2pm\theta_n'(m_2)
\end{align*}
for some  $m_1\in(2(k-1)m, 2km)$ and $m_2\in (p(k-1)m,p(k+1)m)$.
Then
\begin{align*}
R_n(m)=2pm(m_1-m_2)\theta''_n(m_3) 
\end{align*}
for some $m_3\in(m_2, m_1)$. Each of $m_1$, $m_2$ and $m_3$ is comparable to $m$ up to certain constants that depend on $p$. Therefore,
\begin{align}\label{secondderivative}
R_n(m)\geq c m_3^2\theta_n''(m_3)  \hskip 0.5cm \text {for some constant } c>0.
\end{align}

\subsubsection*{Second Derivative}
It remains to understand $\theta_n''(m)$ as $m$ goes to infinity. 
Recall the integration by parts argument at the beginning of the proof. When we take logarithm and two derivatives we get
\begin{align*}
\theta_n(x)&=\log\Theta_n(x)\\
\theta_n(x)&=-\log(x+2)-\cdots-\log(x+n+1) +\log\widetilde{\Phi}_n(x)\\
\theta_n''(x)&=\frac{1}{(x+2)^2}+\cdots+\frac{1}{(x+n+1)^2}+\left(\log\widetilde{\Phi}_n(x)\right)''\\
\theta_n''(x)&\geq \frac{n}{(x+n+1)^2}+\left(\log\widetilde{\Phi}_n(x)\right)''\\
x^2\theta_n''(x)&\geq x^2\frac{n}{(x+n+1)^2}+x^2\left(\log\widetilde{\Phi}_n(x)\right)''~.\\
\end{align*}
By H\"older's inequality, (since $\delta_n$'s are positive functions where we take integrals in $\widetilde{\Phi_n}$) $$\widetilde{\Phi}_n(tu+(1-t)v)\leq \widetilde{\Phi}_n(u)^t\widetilde{\Phi}_n(v)^{1-t}.$$
In other words, $\widetilde{\Phi}_n(x)$ is log-convex\footnote{Since $\delta_n$'s are not necessarily positive on the whole $R$ we can not claim $\Phi_n(x)$ is log-convex.} and $\left(\log\widetilde{\Phi}_n(x)\right)''\geq 0$. Hence we have the simpler inequality that for any $n\geq 1$, 
$$x^2\theta_n''(x)\geq x^2\frac{n}{(x+n+1)^2}.$$
We combine this with \eqref{secondderivative} to conclude that for any $n\geq 1$, 
$$R_n(m)\geq cm_3^2\frac{n}{(m_3+n+1)^2}.$$
Recall $m_3$ is comparable to $m$; hence for any $n$, there exists $M(n)$ such that for all $m>M(n)$, 
$$R_n(m)\geq c' n$$
for some constant $c'>0$. When we take the exponential of both sides and recall the definition of $R_n(m)$, we get 
\begin{align*}
\left(\frac{\Theta_n(2km)}{\Theta_n(2(k-1)m)}\right)^p\frac{\Theta_n(p(k-1)m)}{\Theta_n(p(k+1)m)}=\exp\left(R_n(m)\right)\geq \exp(c' n). 
\end{align*}
By using \eqref{phi-theta} and choosing $m$ sufficiently large, we get 

\begin{align*}
\left(\frac{\Phi(2km)}{\Phi(2(k-1)m)}\right)^p\frac{\Phi(p(k-1)m)}{\Phi(p(k+1)m)}\geq c'' \left(\frac{\Theta_n(2km)}{\Theta_n(2(k-1)m)}\right)^p\frac{\Theta_n(p(k-1)m)}{\Theta_n(p(k+1)m)}
\geq c''\exp(c'\frac{n}{2}) 
\end{align*}
for some additional constant $c''>0$.\\

This means that there exist positive constants $c'$ and $c''$ such that for any $n$, there exists $M'(n)$ so that for all $m>M'(n)$
\begin{align*}
\left(\frac{\Phi(2km)}{\Phi(2(k-1)m)}\right)^p\frac{\Phi(p(k-1)m)}{\Phi(p(k+1)m)}
\geq c''\exp(c'\frac{n}{2}) 
\end{align*}
Finally we let $n$ go to infinity. This gives us the desired blowup and finishes the proof of Proposition \ref{two}.
\end{proof}

$\BPol$ is bounded on $\Ltol$ by definition. When we combine Propositions \ref{one} and \ref{two} we complete the proof of Theorem \ref{main}.




\section{Proof of Theorem \ref{Sobolev}}

Recall that $r(z_1,\cdots,z_n)=|z_1|^2+\cdots+|z_n|^2-1$ is the standard defining function for $\mathbb{B}^n$ and $\mu(z)=\exp\left(\frac{1}{r}\right)$. 
For $k\in\mathbb{N}$, recall $W^k(\mathbb{B}^n,\mu)$ denotes the weighted $L^2$-Sobolev space with the norm 
$$||f||_{k,\mu}^2=\sum_{|\beta+\gamma|\leq k}\int_{\mathbb{B}^n}\left|\frac{\partial^{\beta+\gamma}}{\partial \overline{z}^{\beta}\partial z^{\gamma}}f(z)\right|^2\mu(z)dV(z).$$

For a multi-index $\gamma$, let $d_{\gamma}$ denote the  $L^2$-norm of the monomial $z^{\gamma}$, namely $$d_{\gamma}^2=\int_{\mathbb{B}^n}\left|z^{\gamma}\right|^2\mu(z)dV(z).$$  For the proof of Theorem \ref{Sobolev}, we follow the argument in \cite[Section 4]{Boas84}. We start with the following lemmas. The first one is an asymptotic estimate that replaces a Brunn-Minkowski type inequality.

\begin{lemma}\label{ds} 

There exists a constant $\kappa>0$ such that
\begin{align}\label{dse}
\frac{1}{\kappa}\frac{e^{-2\sqrt{|\gamma|+1}}}{(|\gamma|+1)^{1/3}}\frac{\gamma_1!\cdots\gamma_n!}{(|\gamma|+1)!} \leq d_{\gamma}^2 \leq \kappa \frac{e^{-2\sqrt{|\gamma|+1}}}{(|\gamma|+1)^{1/3}}\frac{\gamma_1!\cdots\gamma_n!}{(|\gamma|+1)!}
\end{align}
for all multi-indices $\gamma$.
\end{lemma}

By using the estimate above, we obtain the following inequality between $L^2$-norms of monomials, which can be thought of as a reverse Cauchy-Schwarz inequality.

\begin{lemma}\label{coeff}
For a given multi-index $\beta$ there exists a constant $K_{\beta}$ such that 
\begin{align}\label{key}
d_{\alpha}d_{\alpha+2\beta}\leq K_{\beta}(d_{\alpha+\beta})^2
\end{align}
for all multi-indices $\alpha$.
\end{lemma}

The next two lemmas are generalizations of similar arguments in \cite{Boas84} to the weighted setting.

\begin{lemma}\label{operator}
For a given multi-index $\beta$ there exists a bounded operator  $$M_{\beta}:L^2_a(\mathbb{B}^n,\mu) \to L^2_a(\mathbb{B}^n,\mu)$$ such that 
$$ \left<h,\frac{\partial^{\beta}}{\partial z^{\beta}}g\right>_{\mu}= \left<\frac{\partial^{\beta}}{\partial z^{\beta}}M_{\beta}h, g\right>_{\mu}$$
for all holomorphic polynomials $h$ and $g\in L^2_a(\mathbb{B}^n,\mu)$.
\end{lemma}

\begin{lemma}\label{parts}
For a given multi-index $\beta$ there exists a constant $K_{\beta}$ such that 
$$\left|\left< \frac{\partial^{\beta}}{\partial z^{\beta}}h,f \right>_{\mu}\right|\leq K_{\beta}||h||_{0,\mu}||f||_{|\beta|,\mu}$$
for all $f\in W^{|\beta|}(\mathbb{B}^n,\mu)$ and all holomorphic polynomials $h$.

\end{lemma}

\subsection{Proof of Theorem \ref{Sobolev}} 
For $j\in\mathbb{N}$, let $S_j$ denote the truncation operator on $L^2_a(\mathbb{B}^n,\mu)$; i.e. for $f(z)=\sum_{\alpha}f_{\alpha}z^{\alpha}$
$$S_jf(z)=\sum_{|\alpha|\leq j}f_{\alpha}z^{\alpha}.$$
Note that $S_j$ is a bounded operator with operator norm 1. Also as $j\to\infty$, $S_jf$ converges to $f$ in norm.

For a given multi-index $|\beta|\leq k$ and $f\in W^k(\mathbb{B}^n,\mu)$ we want to show that 
\begin{align}\label{derivative}
\left\|\frac{\partial^{\beta}}{\partial z^{\beta}}\mathbf{B}_{\mu}f\right\|_{\mu}^2\lesssim ||f||_{k,\mu}^2
\end{align}
where the constant is independent of  $f$. This follows easily when we prove that
\begin{align*}
\left\|S_j\frac{\partial^{\beta}}{\partial z^{\beta}}\mathbf{B}_{\mu}f\right\|_{\mu}^2\lesssim ||f||_{k,\mu}^2
\end{align*}
where the constant is independent of  $j$ and $f$.

Let $h\in L^2_a(\mathbb{B}^n,\mu)$ and $j\in\mathbb{N}$, by using the lemmas above we obtain the following estimate

\begin{align*}
\left| \left<h, S_j \frac{\partial^{\beta}}{\partial z^{\beta}}\mathbf{B}_{\mu}f
 \right>_{\mu}\right|&=\left| \left<S_jh, \frac{\partial^{\beta}}{\partial z^{\beta}}\mathbf{B}_{\mu}f
 \right>_{\mu}\right| ~ \text{by orthogonality of monomials}\\
&= \left| \left<  \frac{\partial^{\beta}}{\partial z^{\beta}}M_{\beta}S_jh,f
 \right>_{\mu}\right| ~ \text{ by Lemma \ref{operator}}\\
&\lesssim ||M_{\beta}S_jh||_{0,\mu}||f||_{k,\mu}~\text{ by Lemma \ref{parts}}\\
&\lesssim ||h||_{0,\mu}||f||_{k,\mu}
\end{align*}
where the constant is independent of $j$ and $f$. This estimate, with the duality of $L^2$ spaces, proves \eqref{derivative} and we conclude the proof of Theorem \ref{Sobolev} modulo proofs of the lemmas.

\subsubsection{Proof of Lemma \ref{parts}} This is essentially a repetition and combination of \cite[Lemma 6.1]{Boas84},  \cite[Lemma 2.1]{Straube86} and \cite[Lemma 3]{ZeytuncuSobolev}. We present the argument here for completeness.

If $f$ is supported on a compact subset of $\mathbb{B}^n$ then the estimate follows trivially. Hence we assume $f$ is supported in a neighborhood of the boundary. We choose this neighborhood around the boundary such that there exist smoothly varying orthonormal holomorphic vector fields $L_1,\cdots,L_n$ such that $L_1,\cdots, L_{n-1}$ and $L_n+\overline{L_n}$ are tangential to the boundary, see \cite[page 292]{Boas84}. In this case, by Cauchy-Riemann equations, for any $1\leq j\leq n$ there exist $c_{ij}\in C^{\infty}(\overline{\mathbb{B}^n})$ such that
$$\frac{\partial}{\partial z_j}h=\left(c_{nj}(L_n+\overline{L_n})+\sum_{i=1}^{n-1}c_{ij}L_i\right)h=\mathcal{L}_jh$$
for any holomorphic polynomial $h$. In other words, in this neighborhood of the boundary derivatives of a holomorphic polynomial can be written as a combination of tangential vector fields.

We know that for a tangential vector field $T$, $T(r)=0$ where $r$ is the standard defining function chosen at the beginning. Moreover, since $\mu(z)=\exp\left(\frac{1}{r}\right)$ we have $T(\mu)=0$ too. This means for $f\in W^1(\mathbb{B}^n,\mu)$ and holomorphic polynomial $h$,
\begin{align*}
\left<T(h),f\right>_{\mu}=\left<T(h)\mu,f\right>=\left<T(h\mu),f\right>=\left<h\mu,\widetilde{T}(f)\right>=\left<h,\widetilde{T}(f)\right>_{\mu}
\end{align*}
where $\widetilde{T}$ is a first order differential operator with $C^{\infty}(\overline{\mathbb{B}^n})$ coefficients. Note that no boundary term appears since $T$ is tangential. 

When we combine these two observations we get,
\begin{align*}
\left< \frac{\partial}{\partial z_j}h, f\right>_{\mu}&=\left< \mathcal{L}_j(h), f\right>_{\mu}~\text{ for some tangential vector field } \mathcal{L}_j\\ 
&=\left< h, \widetilde{\mathcal{L}_j}(f)\right>_{\mu}
\end{align*}
where $\widetilde{\mathcal{L}_j}$ is a first order differential operator with $C^{\infty}(\overline{\mathbb{B}^n})$ coefficients.
Now iterating this several times, we get the following. For any multi-index $\beta$, holomorphic polynomial $h$, and  $f\in W^{|\beta|}(\mathbb{B}^n,\mu)$
\begin{align*}
\left< \frac{\partial^{\beta}}{\partial z^{\beta}}h, f\right>_{\mu}
=\left< h, \widetilde{\mathcal{L}_{\beta}}(f)\right>_{\mu}
\end{align*}
for some differential operator $\widetilde{\mathcal{L}_{\beta}}$ of order $|\beta|$ with $C^{\infty}(\overline{\mathbb{B}^n})$ coefficients. Now, Lemma \ref{parts} follows from the Cauchy-Schwarz inequality.

\subsubsection{Proof of Lemma \ref{operator}} This lemma follows from Lemma \ref{coeff} once we define $M_{\beta}$ as follows. For a monomial $z^{\alpha}$,

$$M_{\beta}(z^{\alpha})=\frac{(\alpha+\beta)!(\alpha+\beta)!}{\alpha!(\alpha+2\beta)!} \frac{d_{\alpha}^2}{d_{\alpha+\beta}^2}z^{\alpha+2\beta}.$$

Then $$\frac{||M_{\beta}(z^{\alpha})||_{\mu}}{||z^{\alpha}||_{\mu}}=\frac{(\alpha+\beta)!(\alpha+\beta)!}{\alpha!(\alpha+2\beta)!}\frac{d_{\alpha}d_{\alpha+2\beta}}{d_{\alpha+\beta}^2}.$$
The first fraction is uniformly bounded since
\begin{align*}
\frac{(\alpha+\beta)!(\alpha+\beta)!}{\alpha!(\alpha+2\beta)!}=\frac{\binom{\alpha+\beta}{\beta}}{\binom{\alpha+2\beta}{\beta}}\leq 1.
\end{align*}
The second fraction is uniformly bounded by Lemma \ref{coeff}. Since monomials form an orthogonal basis for $L^2_a(\mathbb{B}^n,\mu)$ we conclude the proof.\\

We note that the proofs of both Lemmas \ref{parts} and \ref{operator} work not only for the particular $\left(\mathbb{B}^n,\mu\right)$ pair but also for any smooth complete Reinhardt domain and any exponentially decaying weight. However, the estimate \eqref{key} on the coefficients requires more work and it is the heart of the matter. It follows from the estimate \eqref{dse} and we present the proof of \eqref{dse} on the unit ball with the specific weight $\mu$. Readers will note that similar arguments work on some more general domains (e.g. complex ellipsoids). 
However, obtaining this sort of decay rate on the $L^2$-norms of monomials is not immediate on general smooth Reinhardt domains with exponentially decaying weights. We leave the general discussion of all complete Reinhardt domains with exponential weights to a future work.

\subsubsection{Proof of Lemma \ref{ds}}

We start with the following estimate, see \cite[Lemma 2.2]{Dostanic04} and \cite[Lemma 1]{Dostanic07}. As $x\to\infty$,
\begin{align}\label{kappa}
\int_0^1 r^{x}\exp\left(\frac{-1}{1-r}\right) dr\approx x^{-1/3}e^{-2\sqrt{x}}.
\end{align}

We obtain \eqref{dse} by taking integrals in radial coordinates and using the estimate \eqref{kappa}. We go over the details for the case $n=2$, the general case follows similarly. We denote the radial image of $\mathbb{B}^2$, the quarter circle in the first quadrant, by $\mathcal{R}$.

\begin{align*}
d_{\gamma}^2&=\int_{\mathbb{B}^2}|z^{\gamma}|^2\mu(z)dV(z)\\
&=4\pi^2\int_{\mathcal{R}}r_1^{2\gamma_1+1}r_2^{2\gamma_2+1}\exp \left(  \frac{-1}{1-(r_1^2+r_2^2)}\right)dr_2dr_1
\end{align*}
where $z_1=r_1e^{i\theta_1}$ and $z_2=r_2e^{i\theta_2}$. Next, we take the integral over $\mathcal{R}$ in radial coordinates. We set $r_1=R\cos\theta$ and $r_2=R\sin\theta$.
\begin{align*}
d_{\gamma}^2&=4\pi^2\int_0^1\int_0^{\pi/2}(R\cos\theta)^{2\gamma_1+1}(R\sin\theta)^{2\gamma_2+1}\exp \left(  \frac{-1}{1-R^2}\right)RdRd\theta\\
&=4\pi^2\int_0^1 R^{2\gamma_1+2\gamma_2+3}\exp \left(  \frac{-1}{1-R^2}\right)dR
\int_0^{\pi/2}(\cos\theta)^{2\gamma_1+1}(\sin\theta)^{2\gamma_2+1}d\theta\\
&=2\pi^2\int_0^1 R^{|\gamma|+1}\exp \left(  \frac{-1}{1-R}\right)dR
\int_0^{\pi/2}(\cos\theta)^{2\gamma_1+1}(\sin\theta)^{2\gamma_2+1}d\theta
\end{align*}
A simple calculation of trigonometric integrals reveals that
\begin{align*}
\int_0^{\pi/2}(\cos\theta)^{2\gamma_1+1}(\sin\theta)^{2\gamma_2+1}d\theta=\frac{1}{2}\frac{\gamma_1!\gamma_2!}{(|\gamma|+1)!}.
\end{align*}
Therefore, for big enough $|\gamma|$, we invoke \eqref{kappa} to get
\begin{align*}
d_{\gamma}^2&\approx \frac{e^{-2\sqrt{|\gamma|+1}}}{(|\gamma|+1)^{1/3}}\frac{\gamma_1!\gamma_2!}{(|\gamma|+1)!}
\end{align*}
where the constant is independent of the multi-index $\gamma$. This estimate concludes the proof of Lemma \ref{ds}.

Once we have the asymptotic behavior of $L^2$-norms of monomials the next proof follows quickly.

\subsubsection{Proof of Lemma \ref{coeff}} 
We remark again that the inequality \eqref{key} (constant independent of $\alpha$) is not immediate. We plug \eqref{dse} back in \eqref{key},

\begin{align*}
\left(\frac{d_{\alpha}d_{\alpha+2\beta}}{d^2_{\alpha+\beta}}\right)^2&\approx 
\frac{e^{-2\sqrt{|\alpha|+1}}}{(|\alpha|+1)^{1/3}}\frac{\alpha_1!\alpha_2!}{(|\alpha|+1)!}
\frac{e^{-2\sqrt{|\alpha+2\beta|+1}}}{(|\alpha+2\beta|+1)^{1/3}}\frac{(\alpha_1+2\beta_1)!(\alpha_2+2\beta_2)!}{(|\alpha+2\beta|+1)!}\\
&\times \left(\frac{(|\alpha+\beta|+1)^{1/3}}{e^{-2\sqrt{|\alpha+\beta|+1}}}\frac{(|\alpha+\beta|+1)!}{(\alpha_1+\beta_1)!(\alpha_2+\beta_2)!}\right)^2\\
&\approx \exp\left[-2\left(\sqrt{|\alpha|+1}+\sqrt{|\alpha+2\beta|+1}-2\sqrt{|\alpha+\beta|+1}\right)\right]\\
&\times\frac{(|\alpha+\beta|+1)^{2/3}}{(|\alpha|+1)^{1/3}(|\alpha+2\beta|+1)^{1/3}}\\
&\times \frac{\alpha_1!\alpha_2!(\alpha_1+2\beta_1)!(\alpha_2+2\beta_2)!}{(\alpha_1+\beta_1)!(\alpha_1+\beta_1)!(\alpha_2+\beta_2)!(\alpha_2+\beta_2)!}\\
&\times \frac{(|\alpha+\beta|+1)!(|\alpha+\beta|+1)!}{(|\alpha|+1)!(|\alpha+2\beta|+1)!}
\end{align*}

The three fractions after the exponential term are uniformly bounded for all $\alpha$. For the first fraction this is immediate and for the next two we can use the Stirling's formula. If we make sure that the expression inside the exponential function 
\begin{align}\label{mean}
-2\left(\sqrt{|\alpha|+1}+\sqrt{|\alpha+2\beta|+1}-2\sqrt{|\alpha+\beta|+1}\right)
\end{align}
is uniformly bounded
then everything will be uniformly bounded too. Indeed,
\begin{align*}
-2\left(\sqrt{|\alpha|+1}+\sqrt{|\alpha+2\beta|+1}-2\sqrt{|\alpha+\beta|+1}\right)&\leq 2\left(\sqrt{|\alpha+\beta|+1}-\sqrt{|\alpha|+1}\right)\\
&=\frac{2|\beta|}{\sqrt{|\alpha+\beta|+1}+\sqrt{|\alpha|+1}}\\
&\leq 2|\beta|.
\end{align*} 
Therefore \eqref{mean} is uniformly bounded. Hence we obtain \eqref{key} and conclude the proof of Lemma \ref{coeff}.\\

By this we conclude the proof of Theorem \ref{Sobolev}. The proof of the key Lemma \ref{coeff} utilized the asymptotic information \eqref{dse}, instead of a Brunn-Minkowski type inequality. For the analog of Theorem \ref{Sobolev} on convex Reinhardt domains, such an inequality is needed.

\section*{Acknowledgements}

We thank the anonymous referee for the careful reading of the paper and constructive feedback. The proofs of Lemmas \ref{four} and \ref{parts} are greatly improved by the referee's suggestions. Also the exposition of the paper benefited a lot from the referee's comments.

\vskip 1cm
\bibliographystyle{alpha}
\bibliography{LpBib}

\begin{thebibliography}{CDM14}

\bibitem[BG95]{BonamiGrellier}
Aline Bonami and Sandrine Grellier.
\newblock Weighted {B}ergman projections in domains of finite type in {${\bf
  C}^2$}.
\newblock In {\em Harmonic analysis and operator theory ({C}aracas, 1994)},
  volume 189 of {\em Contemp. Math.}, pages 65--80. Amer. Math. Soc.,
  Providence, RI, 1995.

\bibitem[Boa84]{Boas84}
Harold~P. Boas.
\newblock Holomorphic reproducing kernels in {R}einhardt domains.
\newblock {\em Pacific J. Math.}, 112(2):273--292, 1984.

\bibitem[B{\c{S}}12]{SonmezBarrett}
David Barrett and S{\"o}nmez {\c{S}}ahuto{\u{g}}lu.
\newblock Irregularity of the {B}ergman projection on worm domains in {$\Bbb
  C^n$}.
\newblock {\em Michigan Math. J.}, 61(1):187--198, 2012.

\bibitem[CD06]{Charp06}
Philippe Charpentier and Yves Dupain.
\newblock Estimates for the {B}ergman and {S}zeg\"o projections for
  pseudoconvex domains of finite type with locally diagonalizable {L}evi form.
\newblock {\em Publ. Mat.}, 50(2):413--446, 2006.

\bibitem[CDM13]{Charpentier13}
P.~Charpentier, Y.~Dupain, and M.~Mounkaila.
\newblock {E}stimates for {W}eighted {B}ergman {P}rojections on {P}seudo-convex
  {D}omains of {F}inite {T}ype in {$\Bbb{C}^n$}.
\newblock {\em Preprint}, ar{X}iv:1212.1078v3, 2013.

\bibitem[CDM14]{Charpentier14}
P.~Charpentier, Y.~Dupain, and M.~Mounkaila.
\newblock {O}n {E}stimates for {W}eighted {B}ergman {P}rojections.
\newblock {\em Preprint}, ar{X}iv:1403.3412v2, 2014.

\bibitem[CL97]{ChangLi}
Der-Chen Chang and Bao~Qin Li.
\newblock Sobolev and {L}ipschitz estimates for weighted {B}ergman projections.
\newblock {\em Nagoya Math. J.}, 147:147--178, 1997.

\bibitem[Dos04]{Dostanic04}
Milutin~R. Dostani{\'c}.
\newblock Unboundedness of the {B}ergman projections on {$L\sp p$} spaces with
  exponential weights.
\newblock {\em Proc. Edinb. Math. Soc. (2)}, 47(1):111--117, 2004.

\bibitem[Dos07]{Dostanic07}
Milutin~R. Dostani{\'c}.
\newblock Integration operators on {B}ergman spaces with exponential weight.
\newblock {\em Rev. Mat. Iberoam.}, 23(2):421--436, 2007.

\bibitem[FR75]{ForelliRudin}
Frank Forelli and Walter Rudin.
\newblock Projections on spaces of holomorphic functions in balls.
\newblock {\em Indiana Univ. Math. J.}, 24:593--602, 1974/75.

\bibitem[KP07]{KrantzPeloso07Statements}
Steve~G. Krantz and Marco~M. Peloso.
\newblock New results on the {B}ergman kernel of the worm domain in complex
  space.
\newblock {\em Electron. Res. Announc. Math. Sci.}, 14:35--41 (electronic),
  2007.

\bibitem[Kra01]{KrantzSCVbook}
Steven~G. Krantz.
\newblock {\em Function theory of several complex variables}.
\newblock AMS Chelsea Publishing, Providence, RI, 2001.
\newblock Reprint of the 1992 edition.

\bibitem[Lig89]{Ligocka89}
Ewa Ligocka.
\newblock On the {F}orelli-{R}udin construction and weighted {B}ergman
  projections.
\newblock {\em Studia Math.}, 94(3):257--272, 1989.

\bibitem[LS04]{LanzaniStein04}
Loredana Lanzani and Elias~M. Stein.
\newblock Szeg\"o and {B}ergman projections on non-smooth planar domains.
\newblock {\em J. Geom. Anal.}, 14(1):63--86, 2004.

\bibitem[MS94]{McNSte94}
J.~D. McNeal and E.~M. Stein.
\newblock Mapping properties of the {B}ergman projection on convex domains of
  finite type.
\newblock {\em Duke Math. J.}, 73(1):177--199, 1994.

\bibitem[PW90]{Pasternak90}
Zbigniew Pasternak-Winiarski.
\newblock On the dependence of the reproducing kernel on the weight of
  integration.
\newblock {\em J. Funct. Anal.}, 94(1):110--134, 1990.

\bibitem[Str86]{Straube86}
Emil~J. Straube.
\newblock Exact regularity of {B}ergman, {S}zeg{\H o} and {S}obolev space
  projections in nonpseudoconvex domains.
\newblock {\em Math. Z.}, 192(1):117--128, 1986.

\bibitem[Zey13a]{ZeytuncuTran}
Yunus~E. Zeytuncu.
\newblock {$L^p$} regularity of weighted {B}ergman projections.
\newblock {\em Trans. Amer. Math. Soc.}, 365(6):2959--2976, 2013.

\bibitem[Zey13b]{ZeytuncuSobolev}
Yunus~E. Zeytuncu.
\newblock Sobolev regularity of weighted {B}ergman projections on the unit
  disc.
\newblock {\em Complex Var. Elliptic Equ.}, 58(3):309--315, 2013.

\end{thebibliography}

\end{document}